\documentclass[review,onefignum,onetabnum,pdftex]{siamonline190516}

\usepackage{lipsum}
\usepackage{amsfonts}
\usepackage{graphicx}
\usepackage{epstopdf}
\usepackage{algorithmic}
\ifpdf
  \DeclareGraphicsExtensions{.eps,.pdf,.png,.jpg}
\else
  \DeclareGraphicsExtensions{.eps}
\fi

\usepackage{enumitem}
\setlist[enumerate]{leftmargin=.5in}
\setlist[itemize]{leftmargin=.5in}

\newsiamremark{remark}{Remark}
\newsiamremark{hypothesis}{Hypothesis}
\crefname{hypothesis}{Hypothesis}{Hypotheses}
\newsiamthm{claim}{Claim}

\headers{Augmented Newton Method}{Md Sarowar Morshed}

\title{Augmented Newton Method for Optimization: Global Linear Rate and Momentum Interpretation}

\author{Md Sarowar Morshed \thanks{Independent Researcher
  (\email{riponsarowar@outlook.com}, \url{https://morshed.netlify.app/}).}}

\usepackage{amsopn}

\usepackage{ wrapfig}

\usepackage{subcaption}
\usepackage{float}

\usepackage{cite}

\usepackage{amsfonts}
\usepackage{amssymb,amsmath,latexsym,tensor}
\usepackage{mathrsfs}
\usepackage{enumerate}
\usepackage{savesym}
\usepackage{graphicx}
\usepackage{adjustbox}
\usepackage[pdftex]{color}
\usepackage[font=footnotesize,labelfont=bf]{caption}
\savesymbol{AND}
\usepackage{epstopdf}
\usepackage{algorithm}
\usepackage{algorithmic}

\newcommand{\R}{\mathbb{R}}

\DeclareMathOperator*{\argmin}{arg\,min}

\usepackage{mathtools}
\DeclarePairedDelimiterX{\inp}[2]{\langle}{\rangle}{#1, #2}

\newtheorem{assumption}{Assumption}

\newcommand{\eqdef}{\coloneqq} 

\usepackage{xcolor}
\usepackage{mdframed}

\usepackage[colorinlistoftodos,prependcaption,textsize=tiny]{todonotes}

\usepackage[makeroom]{cancel}

\usepackage{array,multirow}

\usepackage{amssymb}

\usepackage[utf8]{inputenc}

\ifpdf
\hypersetup{
  pdftitle={Augmented Newton Method for Optimization: Global Linear Rate and Momentum Interpretation},
  pdfauthor={Md Sarowar Morshed}
}
\fi

\begin{document}

\maketitle

\begin{abstract}
We propose two variants of Newton method for solving unconstrained minimization problem. Our method leverages optimization techniques such as penalty and augmented Lagrangian method to generate novel variants of the Newton method namely the \textit{Penalty Newton} method and the \textit{Augmented Newton} method. In doing so, we recover several well-known existing Newton method variants such as \textit{Damped Newton}, \textit{Levenberg}, and \textit{Levenberg-Marquardt} methods as special cases. Moreover, the proposed \textit{Augmented Newton} method can be interpreted as Newton method with adaptive heavy ball momentum. We provide global convergence results for the proposed methods under mild assumptions that hold for a wide variety of problems. The proposed methods can be sought as the penalty and augmented extensions of the results obtained by Karimireddy \textit{et. al} \cite{karimireddy2018global}. 
\end{abstract}

\begin{keywords}
Newton Method,  Generalized Linear Model, Augmented Newton Method, Penalty Newton Method, Penalty Method, Augmented Lagrangian, Method of Multipliers, Nonlinear Optimization.
\end{keywords}

\begin{AMS}
49M15, 49M37, 58C15, 65K05, 65K10, 65Y20, 68Q25, 90C06, 90C30, 90C51.
\end{AMS}

\section{Introduction}
\label{als:an:sec:intro}
In this work, we consider solving the following unconstrained minimization problem:
\begin{equation}\label{als:an:1} \min_{x\in \R^n} f(x),
\end{equation}
where $f:\R^n\to \R$ is a well behaved function. Newton's method is one of the most important method for solving the above optimization problem. Although, Newton's method is originally proposed in the context of finding roots of polynomial equations, from the last century it has been recognised as one of the fundamental second order methods for solving optimization problems. In the following, we discuss the historical developments of the Newton method as well as the modern developments of Newton method variants.

\paragraph{Historical developments} The first occurrence of Newton type method can be traced back to the work of Persian astronomer and mathematician Al-Kashi (1380–1429). In his seminal work  \textit{'The Key to Arithmetic'} published in 1427, he discussed a variant of Newton method that was based on the earlier works by the famous polymath Al-Biruni (973–1048) and mathematician Al-Tusi (1135-1213). The work of Al-Kashi remained hidden to the western scientific community until the work of Francois Vieta (1540-1603). In 1600, Vieta rediscovered a similar technique like Al-Kashi's in the context of solving scalar polynomial equations of degree six \cite{Historical}. Vieta's method gained considerable attention by the work of Sir Isaac Newton (1643–1727) who improved Vieta's method immensely and developed a closely related method \cite{British} in his seminal works titled \textit{'De analysi per aequationes numero terminorum infinitas'} (1669, published in 1711) and \textit{'De metodis fluxionum et serierum infinitarum'} (1671, published in 1736). Based on the work of Newton, Joseph Raphson (1648-1715) in 1690 proposed a closely related method for finding roots of polynomial equations. However, the connection between Newton method and function derivative was unknown until the later work of Thomas Simpson (1710-1761). In 1740, Simpson proposed a closely related variant of Newton method which was the earliest method that resembles the modern Newton method most \cite{Cajori}.

During early 1900s, the works of Bennet \cite{Bennett592} and Kantarovich \cite{Kantorovich} are the driving force for the popularization of the Newton method as an optimization algorithm. In his seminal work \cite{Kantorovitch1939} Kantarovich was the first one to prove local linear convergence of the Newton method for solving operator equation. This prompted several convergence proofs of variants of Newton method under various assumptions (for a detailed discussion on this issue we refer interested reader to the work of Ortega \textit{et. al} \cite{Ortega}). In the 1980s, Newton's scheme has found applications in various nonlinear optimization algorithm settings such as \textit{Augmented Lagrangian} \cite{BERTSEKAS}, and \textit{Interior Point Methods} (IPMs) \cite{selfc}. In that time, Nesterov and Nemirovski \cite{selfc} proved that under the condition of 'self-concordance' Newton method achieves quadratic convergence rate locally.

\paragraph{Modern developments} In the last decade, there has been a surge of work developing convergence results of variants of Newton method (for a detailed discussion please see \cite{Nesterov2006} and the references therein). At that time, most of the global convergence results either used strong assumptions on the function $f(x)$ or had incomparable convergence rates \footnote{One can obtain similar type of global convergence rates for the inexact type Newton method \cite{Scheinberg2016,Lee2019}.} with respect to the \textit{Gradient Descent} (GD) algorithm. Moreover, some variants have slower rates compared to the vanilla GD algorithm proposed in \cite{Gurbuzbalaban2015}. The first breakthrough result in terms of convergence result has been obtained by Nesterov and Polyak. In their seminal work \cite{Nesterov2006}, they obtained $O(1/k^2)$ convergence rate. Subsequently after that Nesterov achieved $O(1/k^3)$ rate using the so-called cubic regularization technique \cite{Nesterov2008}. The main ingenuity of these works is that, these rates do not require either the strong convexity or the Lipschitz smoothness conditions which were the state-of-the-art assumptions \cite{Polyak2006,Lee2019} of that time. The main disadvantage of these cubic variants is that the cubic subproblem that is the building block of these methods is hard to solve and often times impractical for huge dimensional problems. In their recent work, Karireddy \textit{et. al} provided the convergence results of the Newton method and its trust region variants under the $c$--stability condition \cite{karimireddy2018global}. They showed that the $c$--stability condition is weaker than the standard Lipschitz smoothness, and strong convexity conditions and furthermore $c$--stability is much more weaker than the 'self-concordance' assumption of Nesterov and Nemirovski \cite{selfc}. Recently it has been shown that convergence of some variants of Newton's method can be achieved without smoothness or convexity \cite{roosta2019newtonmr}.

Most modern variants of the Newton method can be categorized into three special types such as \emph{Subsampled Newton}, \emph{Cubic Newton}, and \emph{Subspace Newton} methods. Subsampled Newton methods are stochastic second order algorithms designed mostly to tackle the finite sum minimization problem. Over the years, several variants such as \emph{Subsampled Newton} methods\cite{byrd2011, erdogdu2015, xu2016, Roosta2019} and \emph{Subsampled Cubic Newton} methods \cite{kohler2017, Xu2020, wang2018} and \emph{Stochastic Cubic Newton} methods \cite{tripuraneni2018,Cartis2018, kovalev2019stochastic} have been proposed. These methods make use of the Hessian sketching technique \cite{pilanci2017}. Cubic Newton method make use of the cubic regularization of the Taylor approximation of the function $f(x)$. In their work \cite{Nesterov2006}, Nesterov \textit{et. al} established the first global complexity results of the Cubic Newton method. After their work surprising result, research boomed in developing variants of cubic Newton method such as accelerated \cite{Nesterov2008,monteiro2013}, adaptive \cite{cartis2011adaptive1,cartis2011adaptive2}, block \cite{doikov2018}, and universal \cite{grapiglia2017,grapiglia2019,doikov2019} cubic Newton methods. Subspace Newton methods are the type of stochastic algorithms that use random projections of the Hessian matrix. These types of methods are of the form $x_{k+1} = x_k + S h_k$, where $S$ is a random sketch matrix and $h_k$ is the sketched Newton direction. Several available variants of these types are Stochastic Dual Newton Ascent \cite{qu:2016}, Randomized Subspace Newton \cite{gower2019rsn}, Stochastic Subspace Cubic Newton \cite{hanzely2020}, and Sketched Newton-Raphson \cite{yuan2020sketched}.

\subsection{Newton Method}
\label{als:an:subsec:newton}
The full Newton method for solving problem \eqref{als:an:1} can be expressed as the following update formula:
\begin{align}
\label{als:an:2}
x_{k+1} = x_k -  \left(\nabla^2 f(x_k) \right)^\dagger \nabla f(x_k).
\end{align}
A more general form of the update is the so-called \textit{Damped Newton} (DN) method that can expressed as follows:
\begin{align}
\label{als:an:3}
x_{k+1} = x_k - t  \left(\nabla^2 f(x_k) \right)^\dagger \nabla f(x_k).
\end{align}
An standard way of choosing step size $t$ is by backtracking line search which goes as follows: start with $t = 1$ and $0 < \alpha \leq 1/2$, $0 < \beta < 1$ and check the following:
\begin{align*}
    f(x_{k+1}) > f(x_k) - \alpha t \nabla f(x_k)^{\top} \left(\nabla^2 f(x_k) \right)^\dagger \nabla f(x_k).
\end{align*}
The quantity $\left(\nabla f(x_k)^{\top} \left(\nabla^2 f(x_k) \right)^\dagger \nabla f(x_k) \right)^{1/2}$ is known as the \textit{Newton Decrement}. It is well-known that the above method with backtracking line search achieve quadratic convergence rate under some nice assumptions on function $f(x)$.

\subsection{Contributions}
\label{als:an:subsec:contr}
In this paper, we make the following fundamental contributions:
\begin{itemize}
    \item \textbf{New variants of Newton method.} We propose two variants of Newton method namely the \textit{Penalty Newton Method} (PNM) and \textit{Augmented Newton Method} (ANM) by incorporating the Lagrangian penalty function based techniques on the Newton system \eqref{als:an:17}.
    \item \textbf{Special cases.} From our proposed algorithms we recover well-known algorithms such as \textit{Levenberg} and \textit{Levenberg-Marquardt} algorithms as special cases. 
    \item \textbf{Momentum interpretation.} Moreover, the proposed ANM algorithm can be interpreted as Newton method with adaptive heavy ball momentum.
    \item \textbf{Global linear rate.} We analyze and obtain linear convergence results for the proposed methods under weaker assumptions on the function $f(x)$. For instance, we proved convergence results without imposing the strong convexity and Lipschitz continuity conditions on the function $f(x)$.
\end{itemize}
Furthermore, we analyze and compare the efficiency of the proposed methods to the Newton method for solving \textit{Generalized Linear Models}. The proposed PNM and ANM algorithms bear resemblance to the recently proposed penalty and augmented Kaczmarz methods for solving linear systems and linear feasibility problems \cite{als:kac}.

\section{Notations, Assumptions and Preliminary Results}
\label{als:an:sec:not}

In this section, we introduce some standard notations and definitions that will be used throughout the paper. We also provide necessary assumptions that are crucial to the convergence analysis of the proposed methods. At the end of this section, we mention some technical results most of those are borrowed from the literature.

\subsection{Notations \& Definitions}
For any matrix $A$, by $A^{\dagger}$, $\textbf{Range}(A)$, $\textbf{Null}(A)$,  $\lambda_{\min}^+(A)$, $\lambda_{\max}(A)$, $\sigma_{\min}^+(A)$, $\sigma_{\max}(A)$  we denote the Moore-Penrose pseudo-inverse, range, null, the smallest nonzero eigenvalue, the largest eigenvalue of matrix $A$, the smallest nonzero singular value, the largest singular value of matrix $A$, respectively. For any positive definite matrix $B \in \R^{n \times n}$, we denote $\langle x, B x \rangle = x^{\top}Bx= \|x\|_B^2 $. We denote the optimal value of \eqref{als:an:1} as $f^* = f(x^*)$, where, $x^*$ is the optimal solution. We denote $\nabla f(x)$ and $\nabla^2 f(x) = \mathbf{H}(x)$ as the gradient and Hessian of $f$ at $x$, respectively. Let, $G \in \R^{n \times n}$ be a positive definite matrix, then for any $x \in \R^n$ let us define the following matrices:
\begin{align}
   &  \mathbf{K}(x) = \left(\frac{1}{\rho} G +  \mathbf{H}(x) \right)^{-1},  \ \  \mathbf{L}(x) = \mathbf{H}^{\frac{1}{2}} (x) \left(\frac{1}{\rho} I+ \mathbf{H}^{\frac{1}{2}} (x)  G^{-1}  \mathbf{H}^{\frac{1}{2}} (x) \right)^{-1} \mathbf{H}^{\frac{1}{2}} (x). \label{als:an:4}
\end{align}
We denote the set $\mathcal{Q}_{\rho}$ as the level set of function $f(x) + \frac{L}{2 \rho} \|x-y\|^2_G$ associated with initial iterates $x_0$ and $ y_0$ as follows:
\begin{align}
\label{als:an:6}
    \mathcal{Q}_{\rho} := \{x, y \in \R^n \ : \ f(x) + \frac{L}{2 \rho} \|x-y\|^2_G \leq f(x_0) + \frac{L}{2 \rho} \|x_0-y_0\|^2_G\}.
\end{align}

\subsection{Assumptions}
\label{als:an:subsec:assum}

Throughout the paper, we assume $f: \R^n \rightarrow \R$ is convex, twice  differentiable and bounded from below. Also, we assume the set of minimizers of problem \eqref{als:an:1} is nonempty. Furthermore, we assume the following two assumptions hold.

\begin{assumption}
\label{als:an:a:1}
There exists constant $0 < \mu \leq L$ such that the following relations hold:
\begin{align}
   & f(x) \leq f(y) + \langle \nabla f(y), x-y\rangle + \frac{L}{2} \|x-y\|^2_{\mathbf{H}(y)}, \label{als:an:7} \\
   & f(x) \geq f(y) + \langle \nabla f(y), x-y\rangle + \frac{\mu}{2} \|x-y\|^2_{\mathbf{H}(y)}, \label{als:an:8}
\end{align}
for all $x,y \in \mathcal{Q}_{\rho}$.
\end{assumption} 

In the literature, these constants have been referred to as the relative smoothness constant, $L$ and the relative convexity constant, $\mu$ \cite{karimireddy2018global, gower2019rsn}. One can show that the conditions of assumption \ref{als:an:a:1} are direct consequences of the smoothness and strong convexity conditions. Indeed, from Lemma 2
of \cite{karimireddy2018global}, we have the following relations
\begin{align*}
& \mbox{$L$--smooth + $\mu$--strongly convex} \quad \Rightarrow \quad  \mbox{$c$--stability}, \\
 & \mbox{$c$--stability} \quad \Rightarrow \quad \mbox{$L$--relative smoothness \&  $\mu$--relative convexity}. 
\end{align*}
where, the condition $c$--stability was introduced in \cite{karimireddy2018global}. Specially in \cite{karimireddy2018global}, the authors proved that assuming the $c$--stability of the Hessian matrix is sufficient for proving global linear convergence of the Newton method. This specific assumption of stable Hessian was also used to analyze the statistical characteristics of the logistic regression \cite{bach2010}, and the convergence of \textit{Stochastic Gradient descent} (SGD) for logistic regression \cite{bach2013,bach2014}. We will also assume the following:
\begin{assumption}
\label{als:an:a:2}
For all $x \in \R^n$, we have $ \nabla f(x) \in \textbf{Range}(\mathbf{H}(x))$.
\end{assumption}
If the Hessian is positive definite, assumption \ref{als:an:a:2} holds trivially. Moreover, the above assumption holds generally for the \textit{Generalized Linear Model}.

\subsection{Preliminary Results}
\label{als:an:subsec:prel}
In the following, we will introduce some preliminary technical results that will be used frequently throughout our convergence analysis.

\begin{lemma} 
\label{als:an:lem:1}
(Lemma 9 in \cite{gower2019rsn}) Let $y \in \R^n$, $\kappa>0$ and $\mathbf{H} \in \R^{n\times n}$ be a symmetric positive semi-definite matrix. Let $w \in \textbf{Range} (\mathbf{H})$. Then, we have
\begin{equation}
\label{als:an:9}
 \min_{x \in \R^n}  \langle w, x-y \rangle + \frac{\kappa}{2} \|x-y\|_{\mathbf{H}}^2 =  -\frac{1}{2 \kappa} \|w\|_{\mathbf{H}^\dagger}^2.
\end{equation}
\end{lemma}

\begin{lemma}
\label{als:an:lem:2}
Let Assumptions \ref{als:an:a:1} and \ref{als:an:a:2} holds, then we have the following
\begin{align}
\label{als:an:10}
 f(x_k) - f^* & \leq \frac{1}{2\mu} \|\nabla f(x_k)\|_{\mathbf{H}^\dagger(x_k) }^2.    
\end{align}
\end{lemma}

\begin{proof}
Substituting $y = x_k$ in \eqref{als:an:8}, we get
\begin{align*}
f(x) \overset{ \eqref{als:an:8}}{\geq} f(x_k) + \langle \nabla f(x_k), x-x_k \rangle + \frac{\mu}{2} \|x-x_k\|^2_{\mathbf{H}(x_k)}.
\end{align*}
Minimizing over all $x \in \mathcal{Q}_{\rho}$ and using Lemma \ref{als:an:lem:1} we get
\begin{align*}
f^* -  f(x_k) & \geq \min_{x \in \mathcal{Q}_{\rho}} \left[\langle \nabla f(x_k), x-x_k \rangle + \frac{\mu}{2} \|x-x_k\|^2_{\mathbf{H}(x_k)}\right] \\
&  \geq \min_{x \in \R^n} \left[\langle \nabla f(x_k), x-x_k \rangle + \frac{\mu}{2} \|x-x_k\|^2_{\mathbf{H}(x_k)}\right] \overset{ \eqref{als:an:9}}{=} -\frac{1}{2\mu} \|\nabla f(x_k)\|_{\mathbf{H}^\dagger(x_k) }^2.
\end{align*}
This proves the Lemma.
\end{proof}

Next, we provide the definition of {\em Generalized Linear Model} that satisfy assumptions \ref{als:an:a:1} and \ref{als:an:a:2}. Let $\phi_i: \R \to \R_+$ be twice differentiable functions such that $ u \leq \phi_i''(t) \leq \ell, \ 0 \leq u \leq \ell$ hold for all $i = 1,2,...,m$. Take, $A = [a_1, a_2, ...,a_m] \in \R^{n \times m}$ where $a_i \in \R^n$ for $i=1, \ldots, m$. Now, let us function $f:\R^n \to \R$ is defined as follows:
\begin{align}
\label{als:an:11}
  f(x) = \frac{1}{m} \sum \limits_{i=1}^m \phi_i (a_i^{\top} x) +\frac{\alpha}{2}\|x\|^2,
\end{align}
for some $\alpha > 0$. The problem $\min_{x \in \R^n} f(x)$ is called a {\em Generalized Linear Model} with $L_2$ regularization. The Hessian of the above function is given by
\begin{align}
\label{als:an:12}
  \nabla^2 f(x) =  \mathbf{H}(x)  = \textstyle \frac{1}{m} \sum \limits_{i=1}^m a_ia_i^{\top}\phi_i'' (a_i^{\top} x) + \alpha I =  \frac{1}{m} 
A \Phi'' (A^{\top} x) A^{\top} + \alpha I.
\end{align}
Since, $  \mathbf{H}(x) \succ 0$ for all $x$ assumption \ref{als:an:a:2} is satisfied trivially. The following Lemma was proven in \cite{gower2019rsn}.

\begin{lemma}
\label{als:an:lem:3}
(Proposition 1, \cite{gower2019rsn}) Let $f:\R^n \to \R$ be a Generalized Linear Model with $0 \leq u \leq \ell$. Then, assumption \ref{als:an:a:2} is satisfied with
\begin{align}
\label{als:an:13}
  L = \frac{\ell \ \sigma_{\max}^2(A) +m\alpha}{u \ \sigma_{\max}^2(A)+m\alpha }
\qquad \mbox{and} \qquad
\mu= \frac{u \ \sigma_{\max}^2(A)+m\alpha }{\ell \ \sigma_{\max}^2(A) +m\alpha}. 
\end{align}
\end{lemma}

\begin{lemma}
\label{als:an:lem:4}
(Woodbury Matrix Identity, \cite{henderson1981}) Assume matrices $A \in \R^{q \times q}$, and $C \in \R^{r \times r }$ are invertible. Then the following identity holds:
\begin{align}
\label{als:an:14}
\left(A+ UCV\right)^{-1} = A^{-1} - A^{-1}  U \left(C^{-1}+VA^{-1} U\right)^{-1} VA^{-1},
\end{align}
for any matrices $U \in \R^{q \times r}$ and $V \in \R^{r \times q}$ such that $A+ UCV$ and $C^{-1}+VA^{-1} U$ are non-singular.
\end{lemma}

\begin{lemma}
\label{als:an:lem:5}
The following relations hold:
\begin{align}
    & \mathbf{H}(x)   \mathbf{K}(x)  = I - \frac{1}{\rho} G  \mathbf{K}(x), \label{als:an:15} \\
    & G \mathbf{K}(x) G = \rho G - \rho \mathbf{L}(x). \label{als:an:16}
\end{align}
\end{lemma}

\begin{proof}
Since, $\mathbf{K}(x)$ is invertible, we have the following simplification:
\begin{align*}
   \mathbf{H}(x)  \mathbf{K}(x) = \left[\frac{1}{\rho} G+ \mathbf{H}(x) - \frac{1}{\rho} G\right]  \mathbf{K}(x)=  \left[\mathbf{K}^{-1}(x)- \frac{1}{\rho} G\right]  \mathbf{K}(x) = I - \frac{1}{\rho} G  \mathbf{K}(x).
\end{align*}
This resolves the first part of Lemma \ref{als:an:lem:5}. This proves the first part of the Lemma. As $G^{-1} \succ 0$ follows from our construction, using the Woodburry matrix identity and the definition of $\mathbf{L}(x) $, we have the following:
\begin{align*}
 \mathbf{L}(x) & =   \mathbf{H}^{\frac{1}{2}} (x)  \left(\frac{1}{\rho}+ \mathbf{H}^{\frac{1}{2}} (x) G^{-1} \mathbf{H}^{\frac{1}{2}} (x) \right)^{-1} \mathbf{H}^{\frac{1}{2}} (x) \nonumber \\
 & \overset{ \eqref{als:an:14}}{=}  \mathbf{H}^{\frac{1}{2}} (x)   \left[ \rho -  \rho^2  \mathbf{H}^{\frac{1}{2}} (x) \left( G+  \rho \mathbf{H}^{\frac{1}{2}} (x)  \mathbf{H}^{\frac{1}{2}} (x) \right)^{-1} \mathbf{H}^{\frac{1}{2}} (x)   \right] \mathbf{H}^{\frac{1}{2}} (x) \\
 & = \mathbf{H}^{\frac{1}{2}} (x)  \left[ \rho -  \rho  \mathbf{H}^{\frac{1}{2}} (x) \left(\frac{1}{\rho} G+  \mathbf{H}^{\frac{1}{2}} (x) \mathbf{H}^{\frac{1}{2}} (x) \right)^{-1} \mathbf{H}^{\frac{1}{2}} (x) \right] \mathbf{H}^{\frac{1}{2}} (x) \\
 & = \mathbf{H}^{\frac{1}{2}} (x)  \left[ \rho -  \rho  \mathbf{H}^{\frac{1}{2}} (x)  \mathbf{K}(x)  \mathbf{H}^{\frac{1}{2}} (x) \right] \mathbf{H}^{\frac{1}{2}} (x)  \\
 & = \rho  \mathbf{H} (x)  - \rho  \mathbf{H} (x) \mathbf{K}(x)   \mathbf{H} (x)  \nonumber \\
 & \overset{ \eqref{als:an:15}}{=} \rho  \mathbf{H} (x)  - \rho  \mathbf{H} (x)  \mathbf{K}(x) \left[\mathbf{K}^{-1}(x) - \frac{1}{\rho} G\right]  =  \mathbf{H} (x)  \mathbf{K}(x) G = G - \frac{1}{\rho} G \mathbf{K}(x) G.
\end{align*}
With further simplification, we get the second part of Lemma \ref{als:an:lem:5}.
\end{proof}

\begin{lemma}
\label{als:an:lem:6}
(Lemma 10, \cite{gower2019rsn}) Assume, $N$ is a positive semidefinite matrix and matrix $M$ such that $\textbf{Null}(N) \subset \textbf{Null}(M^{\top} )$ holds. Then, we have the following:
\begin{equation*}
\textbf{Null}(M) = \textbf{Null}(M^{\top}N M) \quad  \text{and} \quad \textbf{Range}(M^{\top} ) = \textbf{Range}(M^{\top}N M).
\end{equation*}
\end{lemma}

\section{Lagrangian Penalty Approaches}
\label{als:an:sec:lp}
In this section we introduce the proposed schemes that are constructed based on Lagrangian penalty methods. To achieve this, let us rewrite the Newton update provided in \eqref{als:an:3} as the solution of the following optimization problem:
\begin{equation}\label{als:an:17} x_{k+1} \in    \argmin_x \frac{1}{2}\|x-x_k\|^2_{ \mathbf{H}(x_k)} \quad \textbf{s.t} \quad   \mathbf{H}(x_k) (x-x_k) = - \frac{1}{L}  \nabla f(x_k),
\end{equation}
here, we use $t = 1/L$. We interpret the above problem as the projected Newton system. In the following, we will discuss the proposed variants of the Newton method that are achieved by incorporating various modern optimization techniques to the above Newton system.
\subsection{Penalty Newton Method (PNM)} Our first approach builds on the idea of penalty method that is used frequently for solving constrained optimization problems. The basic premise of the PNM algorithm is as follows: we modify the objective function of problem \eqref{als:an:17} by adding a penalty term \footnote{For simplicity of exposition, in this work we consider only quadratic penalty function.} to mitigate the constraint violation, then we solve the resulting unconstrained optimization problem to find the next update. Our end goal is that by following the above construction, we can make the algorithm satisfy the equality constraint by increasing the penalty of not doing that. With the above setup, we get the following optimization problem:
\begin{align}
\label{als:an:18}
& x_{k+1} \in   \argmin_x \mathcal{L} (x, \rho)  \eqdef \frac{1}{2}\|x-x_k\|^2_{ \mathbf{H}(x_k)} + \frac{\rho}{2}  \big \|  \mathbf{H}(x_k) (x-x_k) + \frac{1}{L} \nabla f(x_k) \big \|^2_{G^{-1}},
\end{align}
here, $G \succ 0$ is a positive definite matrix. We derive a closed form solution in the next Lemma.

\begin{lemma}
\label{als:an:lem:7}
The following update formula solves the optimization problem of \eqref{als:an:18}:
\begin{align}
\label{als:an:19}
    x_{k+1} = x_k - \frac{1}{L} \left(\frac{1}{\rho} G +   \mathbf{H}(x_k) \right)^{-1} \nabla f(x_k)  = x_k - \frac{1}{L} \mathbf{K}(x_k)  \nabla f(x_k).
\end{align}
\end{lemma}

\begin{proof}
Setting $\frac{\partial \mathcal{L} (x, \rho)}{\partial x} = 0$, we get
\begin{align}
\label{als:an:20}
    0 = \frac{\partial \mathcal{L} (x, \rho)}{\partial x} & = \left[ \mathbf{H}(x_k)+ \rho \  \mathbf{H}(x_k) G^{-1}  \mathbf{H}(x_k) \right] (x-x_k)  + \frac{\rho}{L}  \mathbf{H}(x_k) G^{-1} \nabla f(x_k).
\end{align}
This is the first order optimality condition of problem \eqref{als:an:18}. Next, we show that the update formula of \eqref{als:an:19} satisfies the above optimatliy condition. Using the update formula, we have the following:
\begin{align}
\label{als:an:21}
   \left[ \mathbf{H}(x_k)+ \rho  \mathbf{H}(x_k) G^{-1}  \mathbf{H}(x_k) \right] & (x_{k+1}-x_k) \overset{ \eqref{als:an:19}}{=}  -  \frac{1}{L} \left[ \mathbf{H}(x_k)+ \rho  \mathbf{H}(x_k) G^{-1}  \mathbf{H}(x_k) \right]  \mathbf{K}(x_k) \nabla f(x_k)  \nonumber \\
   =  - & \frac{1}{L}  \mathbf{H}(x_k)   \mathbf{K}(x_k) \nabla f(x_k) - \frac{\rho}{L}  \mathbf{H}(x_k) G^{-1}  \mathbf{H}(x_k)  \mathbf{K}(x_k)  \nabla f(x_k) \nonumber \\
   \overset{ \eqref{als:an:15}}{=}   - & \frac{1}{L}  \mathbf{H}(x_k)   \mathbf{K}(x_k)  \nabla f(x_k) - \frac{\rho}{L}  \mathbf{H}(x_k) G^{-1} \left( I - \frac{1}{\rho}  G \mathbf{K}(x_k) \right)   \nabla f(x_k) \nonumber \\
   = - & \frac{\rho}{L}  \mathbf{H}(x_k) G^{-1} \nabla f(x_k).
\end{align}
Here, we used the identity, $\mathbf{H}(x_k)   \mathbf{K}(x_k)  = I - \frac{1}{\rho}  G \mathbf{K}(x_k)$ from Lemma \ref{als:an:lem:5}. Substituting the identity of \eqref{als:an:21} into \eqref{als:an:20}, we deduce that $x_{k+1}$ satisfies the first order optimality condition. Therefore, we can say that $x_{k+1}$ provided in \eqref{als:an:19} is a solution of the optimization problem \eqref{als:an:18}.
\end{proof}

 \begin{algorithm}
\caption{$x_{K+1} = \textbf{PNM}(G, c, f(x), K)$}
\label{als:an:alg:pnm}
\begin{algorithmic}
\STATE{Choose initial points $x_0 \in \R^n,  \ \rho_0 \in \R$}
\WHILE{$k \leq K$}
\STATE{\begin{equation*}
     x_{k+1} = x_k -  \frac{1}{L}  \left(\frac{1}{\rho_k} G +  \mathbf{H}(x_k) \right)^{-1}  \nabla f(x_k), \quad \quad   \rho_{k+1} = c \rho_k;
\end{equation*}
$k \leftarrow k+1$;}
\ENDWHILE
\end{algorithmic}
\end{algorithm}

\subsection{Augmented Newton Method (ANM)}
\label{als:an:subsec:als}

Our second approach uses the idea of the augmented Lagrangian method for optimization. The method was proposed by Hestenes in 1969 \cite{Hestenes1969} and was originally called as the method of multipliers (see the works of Powell \cite{powell} and Bertsekas \cite{BERTSEKAS} for a detailed discussion). The basic premise of the ANM algorithm is as follows: we solve problem \eqref{als:an:17} by the augmented Lagrangian method, unlike the penalty method here we update the dual solution $z$ as an estimator of Lagrange multiplier. The major advantage of this approach is that it is not required to make $\rho \rightarrow \infty$ like we did in the penalty method. Instead, we update the corresponding dual solution at each iteration that compensates the smaller penalty. Considering the above framework, we reformulate problem \eqref{als:an:17} using augmented Lagrangian penalty function as follows:
\begin{align}
\label{als:an:22}
x_{k+1} \in  & \argmin_{x} \frac{1}{2}\|x-x_k\|^2_{ \mathbf{H}(x_k)} +  \frac{\rho}{2} \Big \|\mathbf{H}(x_k) (x-x_k) + \frac{1}{L} \nabla f(x_k) \Big\|^2_{G^{-1}} \nonumber \\
& \quad  \qquad \qquad   \textbf{subject to} \qquad   \mathbf{H}(x_k) (x-x_k) + \frac{1}{L} \nabla f(x_k) = 0.
\end{align}
The Lagrangian of the above problem is given by
\begin{align}
\label{als:an:23}
    \mathcal{L} (x,z, \rho) = \frac{1}{2}\|x-x_k\|^2_{ \mathbf{H}(x_k)} & + z^{\top} \left[\mathbf{H}(x_k) (x-x_k) + \frac{1}{L} \nabla f(x_k)\right] \nonumber \\
    & + \frac{\rho}{2} \Big \|\mathbf{H}(x_k) (x-x_k) + \frac{1}{L} \nabla f(x_k) \Big \|^2_{G^{-1}}.
\end{align}
Then, we set the ANM update formulas as follows:
\begin{align}
\label{als:an:24}
    x_{k+1} \in \argmin_{x} \mathcal{L} (x,z_k, \rho), \quad z_{k+1} = z_k + \rho \ G^{-1} \left[\mathbf{H}(x_k) (x_{k+1}-x_k) + \frac{1}{L} \nabla f(x_k)\right].
\end{align}

\begin{algorithm}
\caption{$x_{K+1} = \textbf{ANM}(G,c, f(x), K)$}
\label{als:an:alg:anm}
\begin{algorithmic}
\STATE{Choose initial points $x_0, x_1 \in \R^n, \ \rho_0 \in \R$}
\WHILE{$k \leq K$}
\STATE{\begin{equation*}
    x_{k+1} = x_k - \left(\frac{1}{\rho_k} G +  \mathbf{H}(x_k)\right)^{-1} \left[\frac{1}{L}\nabla f(x_k) - \frac{1}{\rho_k} G (x_k-x_{k-1}) \right], \quad  \rho_{k+1} = c \rho_k;
\end{equation*}
$k \leftarrow k+1$;}
\ENDWHILE
\end{algorithmic}
\end{algorithm}

\begin{lemma}
\label{als:an:lem:8}
Denote, $ u(x_k, z_k) = \frac{\rho}{L} \nabla f(x_k) + G z_k$. Then, the following update formulas solve the Lagrangian system of \eqref{als:an:24}:
\begin{equation} 
\label{als:an:25}
    x_{k+1} = x_k - \frac{1}{\rho}  \mathbf{K}(x_k) \ u(x_k, z_k)  = x_k -    z_{k+1}, \quad  z_{k+1} = \frac{1}{\rho} \mathbf{K}(x_k) \ u(x_k, z_k).
\end{equation}
Moreover, the following update solves the Lagrangian system of \eqref{als:an:24}:
\begin{align}
\label{als:an:26}
    x_{k+1} & = x_k - \left(\frac{1}{\rho} G +  \mathbf{H}(x_k)\right)^{-1} \left[\frac{1}{L}\nabla f(x_k) - \frac{1}{\rho} G (x_k-x_{k-1}) \right].
\end{align}
\end{lemma}

\begin{proof}
Setting $\frac{\partial \mathcal{L} (x,z_k, \rho)}{\partial x} = 0$, we get
\begin{align}
\label{als:an:27}
    0 = \frac{\partial \mathcal{L} (x,z_k, \rho)}{\partial x}  & = \left[ \mathbf{H}(x_k)+ \rho \mathbf{H}(x_k) G^{-1} \mathbf{H}(x_k) \right] (x-x_k)  +  \mathbf{H}(x_k) \left(\frac{\rho}{L} G^{-1} \nabla f(x_k) + z_k \right). 
\end{align}
Using the value of $x_{k+1}$ from \eqref{als:an:25} in place of $x$, we get the following:
\allowdisplaybreaks{\begin{align}
    \left[ \mathbf{H}(x_k)+ \rho \mathbf{H}(x_k) G^{-1} \mathbf{H}(x_k)   \right] & (x_{k+1}-x_k)  \overset{ \eqref{als:an:25}}{=}   - \frac{1}{\rho}\left[ \mathbf{H}(x_k)+ \rho \mathbf{H}(x_k) G^{-1} \mathbf{H}(x_k)  \right] \mathbf{K}(x_k) u(x_k, z_k)  \nonumber \\
  & \overset{ \eqref{als:an:15}}{=}  - \frac{1}{\rho}  \mathbf{H}(x_k) \mathbf{K}(x_k) u(x_k)  -  \mathbf{H}(x_k) G^{-1} \left(I- \frac{1}{\rho} G  \mathbf{K}(x_k) \right) u(x_k, z_k)  \nonumber \\
  & = -  \mathbf{H}(x_k)  \left(\frac{\rho}{L}  G^{-1} \nabla f(x_k)  + z_k \right). \label{als:an:28}
\end{align}}
Substituting the expression in \eqref{als:an:27}, we get that $x_{k+1}$ satisfies the equation \eqref{als:an:27}. Now, from the definition of $z_{k+1}$, we have
\begin{align*}
z_{k+1} & \overset{ \eqref{als:an:24}}{=}   z_k + \rho   G^{-1} \left[ \mathbf{H}(x_k) (x_{k+1}-x_k) + \frac{1}{L} \nabla f(x_k)\right] \nonumber \\
 & \overset{ \eqref{als:an:25}}{=} z_k + \rho  G^{-1} \left[- \frac{1}{\rho}   \mathbf{H}(x_k)  \mathbf{K}(x_k) u(x_k, z_k) + \frac{1}{L} \nabla f(x_k)\right] \nonumber \\
 & \overset{ \eqref{als:an:15}}{=} z_k + \frac{\rho}{L}  G^{-1} \nabla f(x_k) - G^{-1} \left(I- \frac{1}{\rho}  G \mathbf{K}(x_k) \right)  u(x_k, z_k) = \frac{1}{\rho} \mathbf{K}(x_k) \ u(x_k, z_k).
\end{align*}
That is precisely the expression provided in \eqref{als:an:25}. This proves the first part of the Lemma. To prove the second part, we will simplify the update formula further by cancelling the $z$ variables from the update formulas. Since, $z_{k+1} = x_k -x_{k+1}$ and $z_k = x_{k-1}-x_k$, we can combine the update formulas as follows:
\begin{align}
\label{als:an:29}
    x_{k+1}  \overset{ \eqref{als:an:24}}{=} & x_k - \left(\frac{1}{\rho} G +  \mathbf{H}(x_k)\right)^{-1} \left[\frac{1}{L}\nabla f(x_k) + \frac{1}{\rho} G z_k \right] \nonumber \\
     = \ &  \ x_k - \left(\frac{1}{\rho} G +  \mathbf{H}(x_k)\right)^{-1} \left[\frac{1}{L}\nabla f(x_k) - \frac{1}{\rho} G (x_k-x_{k-1}) \right].
\end{align}
This proves the second part.
\end{proof}

\subsection{Special Cases}
\label{als:an:sec:special}
In this subsection, we discuss some special cases that can be obtained from the proposed PNM and ANM algorithms by varying parameters $G$ and $\rho$.

\paragraph{Newton Method} By taking $\rho \rightarrow \infty$ in both the PNM and ANM algorithms, we get the Newton method. In the following Lemma, we provide a detailed discussion about the transformation.

\begin{lemma}
\label{als:an:lem:9}
In the limiting case, both PNM and ANM resolves into the Newton method, i.e., $\lim_{\rho \rightarrow \infty} \text{PNM} \equiv  \lim_{\rho \rightarrow \infty} \text{ANM}  \equiv \text{NM}$.
\end{lemma}

\begin{proof}
From the update formula of the PNM method, we have
\begin{align}
\label{als:an:30}
 \lim_{\rho \rightarrow \infty} x_{k+1}  & \overset{ \eqref{als:an:19}}{=} x_k - \frac{1}{L}  \lim_{\rho \rightarrow \infty} \mathbf{K}(x_k)  \nabla f(x_k) \nonumber \\
 & \overset{ \text{Assumption} \ \ref{als:an:a:2}}{=} x_k - \frac{1}{L}  \lim_{\rho \rightarrow \infty} \mathbf{K}(x_k)  \mathbf{H}^{\frac{1}{2}}(x_k) \mathbf{H}^{\dagger/2}(x_k)  \nabla f(x_k). 
\end{align}
In the last line, we used Assumption \ref{als:an:a:2}. Denote, $Y_k = \mathbf{H}^{\frac{1}{2}}(x_k) G^{-\frac{1}{2}}$, then we have
\allowdisplaybreaks{\begin{align}
\label{als:an:31}
 \lim_{\rho \rightarrow \infty}  x_{k+1} &  \overset{ \eqref{als:an:30}}{=} x_k - \frac{1}{L} G^{-\frac{1}{2}}  \lim_{\rho \rightarrow \infty} \left(\frac{1}{\rho}I+ Y_k^{\top} Y_k\right)^{-1} Y^{\top}_k \mathbf{H}^{\dagger/2}(x_k)  \nabla f(x_k) \nonumber \\
 & \overset{ \eqref{als:an:310}}{=} x_k - \frac{1}{L} G^{-\frac{1}{2}}  Y^{\dagger}_k \mathbf{H}^{\dagger/2}(x_k)  \nabla f(x_k) \overset{ \eqref{als:an:310}}{=} x_k - \frac{1}{L} G^{-\frac{1}{2}}   (Y^{\top}_kY_k)^{\dagger} Y^{\top}_k \mathbf{H}^{\dagger/2}(x_k)  \nabla f(x_k) \nonumber \\ 
 & = x_k - \frac{1}{L}  G^{\frac{1}{2}}\left( G^{\frac{1}{2}} \mathbf{H}(x_k) G^{\frac{1}{2}} \right)^{\dagger}  G^{\frac{1}{2}}  \mathbf{H}^{\frac{1}{2}}(x_k) \mathbf{H}^{\dagger/2}(x_k)   \nabla f(x_k) \nonumber \\
 & \overset{ \text{Assumption} \ \ref{als:an:a:2}}{=}  x_k - \frac{1}{L}   \mathbf{H}^{\dagger}(x_k)    \nabla f(x_k).
\end{align}}
Here, we used Assumption \ref{als:an:a:2} along with the following identities of matrix pseudo-inverse:
\begin{equation} \label{als:an:310} Y^{\dagger}_k =  (Y_k^{\top}Y_k)^{\dagger} Y_k^{\top} =  \lim_{\rho \rightarrow \infty} \left(\frac{1}{\rho}I+ Y^{\top}_k Y_k\right)^{-1} Y_k^{\top}.
\end{equation}
The update formula of \eqref{als:an:31} is precisely the update formula of the Newton method. Moreover, from the update formula of the ANM method, we get,
\begin{align}
\label{als:an:32}
 \lim_{\rho \rightarrow \infty} x_{k+1} & \overset{ \eqref{als:an:29}}{=} x_k -  \lim_{\rho \rightarrow \infty} \mathbf{K}(x_k) \left[\frac{1}{L}\nabla f(x_k) - \frac{1}{\rho} G (x_k-x_{k-1}) \right]  = x_k - \frac{1}{L}  \lim_{\rho \rightarrow \infty} \mathbf{K}(x_k)  \nabla f(x_k).
\end{align}
Here we used the limit $\lim_{\rho \rightarrow \infty} \left[\frac{1}{L}\nabla f(x_k) - \frac{1}{\rho} G (x_k-x_{k-1}) \right] = \frac{1}{L}  \nabla f(x_k)$. This means the  $x$ update sequences of the two methods are same in the limiting case. This proves the Lemma.
\end{proof}

\paragraph{Levenberg $ \&$ Augmented Levenberg Algorithms} With the choice $G = I$, the PNM algorithm resolves into the original \textit{Levenberg} algorithm, i.e.,
\begin{align}
\label{als:an:33}
    x_{k+1} & \overset{ \eqref{als:an:19}}{=} x_k - \frac{1}{L} \left(\frac{1}{\rho} I +  \mathbf{H}(x_k)\right)^{-1} \nabla f(x_k).
\end{align}
Similarly, substituting $G = I$ in the ANM algorithm, we get the following update:
\begin{align}
\label{als:an:34}
    x_{k+1} & \overset{ \eqref{als:an:29}}{=} x_k - \frac{1}{L} \left(\frac{1}{\rho} I +  \mathbf{H}(x_k)\right)^{-1} \nabla f(x_k)  + \frac{1}{\rho} \left(\frac{1}{\rho} I +  \mathbf{H}(x_k)\right)^{-1} (x_k-x_{k-1}).
\end{align}
We refer to the above method as the \textit{Augmented Levenberg} algorithm.

\paragraph{Levenberg-Marquardt $ \&$ Augmented Levenberg-Marquardt Algorithms} With the choice $G = \textbf{diag}(\mathbf{H}(x_k))$ in the PNM algorithm we get the following update:
\begin{align}
\label{als:an:35}
    x_{k+1} & \overset{ \eqref{als:an:19}}{=} x_k - \frac{1}{L} \left(\frac{1}{\rho} \textbf{diag}(\mathbf{H}(x_k)) +  \mathbf{H}(x_k)\right)^{-1} \nabla f(x_k).
\end{align}
This is precisely the \textit{Levenberg-Marquardt} variant of the Newton method. Similarly, substituting $G = \textbf{diag}(\mathbf{H}(x_k))$ in the ANM algorithm, we get the following update:
\begin{align}
\label{als:an:36}
    x_{k+1} & \overset{ \eqref{als:an:29}}{=} x_k - \frac{1}{L} \left(\frac{1}{\rho} \mathbf{D}_k +  \mathbf{H}(x_k)\right)^{-1} \nabla f(x_k)  + \frac{1}{\rho} \left(\frac{1}{\rho} \mathbf{D}_k +  \mathbf{H}(x_k)\right)^{-1} \mathbf{D}_k (x_k-x_{k-1}).
\end{align}
here, we denote $\mathbf{D}_k = \textbf{diag}(\mathbf{H}(x_k)) $. We refer to the above method as the \textit{Augmented Levenberg-Marquardt} algorithm. Using the above parameter values in Theorems \ref{als:an:thm:1} and \ref{als:an:thm:2}, we get the respective convergence results of the above-mentioned methods.

\paragraph{Newton Method With Momentum} Next, we discuss the ANM algorithm from a heavy ball/Polyak momentum perspective. Note that, from the simplified update formula of the ANM algorithm provided in \eqref{als:an:29}, we get
\begin{align}
    x_{k+1} & \overset{ \eqref{als:an:29}}{=} x_k - \frac{1}{L} \left(\frac{1}{\rho} G +  \mathbf{H}(x_k)\right)^{-1} \nabla f(x_k)  + \frac{1}{\rho} \left(\frac{1}{\rho} G +  \mathbf{H}(x_k)\right)^{-1} G (x_k-x_{k-1}) \nonumber \\
  \quad \quad \quad  & = x_k - \frac{1}{L} \left(\frac{1}{\rho} G +  \mathbf{H}(x_k)\right)^{-1} \nabla f(x_k)  + \Theta (x_k) (x_k-x_{k-1}). \label{als:an:37}
\end{align}
The above update can be interpreted as Newton method with adaptive heavy ball/Polyak momentum with $\Theta (x_k) = \frac{1}{\rho} \left(\frac{1}{\rho} G +  \mathbf{H}(x_k)\right)^{-1} G$. To the best of our knowledge this is the first variant of Newton method that incorporates the heavy ball/Polyak momentum update to the Newton method.

\paragraph{Penalty \& Augmented Newton Method for Root Finding} Let's take the equation $f(x) = 0, \ f: \R \rightarrow \R$. Denote $f^{'}(x)$ as the function derivative of $f$, then the proposed methods take the following form:
\begin{align*}
   & \textbf{Penalty Newton: } \quad x_{k+1} = x_k - \frac{\rho \ f(x_k)}{1+ \rho \ f^{'}(x_k)} \\
   & \textbf{Augmented Newton: } \quad x_{k+1} = x_k - \frac{\rho \ f(x_k)}{1+ \rho \ f^{'}(x_k)} +  \frac{ x_k - x_{k-1}}{1+ \rho \ f^{'}(x_k)}\\
\end{align*}

\begin{remark}
\label{als:an:rem:1}
Note that in Algorithms \ref{als:an:alg:pnm} and \ref{als:an:alg:anm}, we provided adaptive versions of the above methods. In that regard, we gradually increase the penalty parameter $\rho$ in a way such that $\rho_k$ grows larger as the iteration progresses. We used the simplest update possible to achieve this, i.e., $\rho_{k+1} = c \rho_k$ and replace $\rho$ with $\rho_k$ is the PNM and ANM methods. 
\end{remark}

\section{Main Results}
\label{als:an:sec:main}
We now present the convergence results for the proposed variants of Newton method.  In brief, for the PNM algorithm we show that $f(x_k) \rightarrow f^* $ holds and for the ANM method we show that both  $f(x_k) \rightarrow f^* $ and $ \|x_k-x_{k-1}\|^2_G \rightarrow 0$ hold. To achieve this we define the following sequence:
\begin{align}
\label{als:an:38}
& \mathcal{V}_k(\rho)    = f(x_k) -f(x^*) + \frac{L}{2 \rho} \|x_k-x_{k-1}\|^2_G, 
\end{align}
for any $k \geq 1$, and $\rho > 0$. We will show that for the ANM algorithm, $\mathcal{V}_k(\rho) $ is a Lyapunov function. Before we delved into the convergence results, let us provide some technical Lemmas that are crucial to the convergence analysis. First, define the following constant:
\begin{align}
\label{als:an:39}
    & \xi(x) = \min_{v \in \textbf{Range}(\mathbf{H}(x))} \frac{\langle \mathbf{H}^{\frac{1}{2}} (x) \mathbf{K}(x)  \mathbf{H}^{\frac{1}{2}} (x) v, v \rangle}{\|v\|^2}, \quad \ \xi = \min_{x \in \mathcal{Q}_{\rho}} \xi(x).
\end{align}
Constant $\xi(x)$ can be interpreted as the condition number of matrix $\mathbf{H}^{\frac{1}{2}} (x) \mathbf{K}(x)  \mathbf{H}^{\frac{1}{2}} (x)$.

\begin{lemma}
\label{als:an:lem:10}
For any $x \in \R^n$, matrices $\mathbf{H}^{\frac{1}{2}} (x) \mathbf{K}(x)  \mathbf{H}^{\frac{1}{2}} (x)$ and $G^{-\frac{1}{2}} \mathbf{L}(x) G^{-\frac{1}{2}}$ have the same nonzero eigenvalues, i.e., $\lambda_{i}^+ \left(\mathbf{H}^{\frac{1}{2}} (x) \mathbf{K}(x)  \mathbf{H}^{\frac{1}{2}} (x)\right) = \lambda_{i}^+ \left(G^{-\frac{1}{2}} \mathbf{L}(x) G^{-\frac{1}{2}}\right)$.
\end{lemma}

\begin{proof}
Let, $Y = \mathbf{H}^{\frac{1}{2}}(x) G^{-\frac{1}{2}}$. Assume $Y$ has the following singular value decomposition, i.e., $ Y = U \Sigma V^{\top}$. Then we have the following simplifications:
\begin{align}
     \mathbf{H}^{\frac{1}{2}} (x) \mathbf{K}(x)  \mathbf{H}^{\frac{1}{2}} (x) & \overset{ \eqref{als:an:4}}{=} Y \left(\frac{1}{\rho} I + Y^{\top}Y \right)^{-1} Y^{\top} = U \Sigma  \left(\frac{1}{\rho} I + \Sigma^{\top}\Sigma \right)^{-1} \Sigma^{\top} U^{\top}, \label{als:an:40} \\
    G^{-\frac{1}{2}} \mathbf{L}(x) G^{-\frac{1}{2}} & \overset{ \eqref{als:an:4}}{=} Y^{\top} \left(\frac{1}{\rho} I + YY^{\top} \right)^{-1} Y  = V \Sigma^{\top}  \left(\frac{1}{\rho} I + \Sigma\Sigma^{\top} \right)^{-1} \Sigma V^{\top}. \label{als:an:41}
\end{align}
Since, $U$ and $V$ are orthogonal matrices from the above expressions we get the following:
\begin{align}
 \label{als:an:42}
    \lambda_{i}^+ \left(\mathbf{H}^{\frac{1}{2}} (x) \mathbf{K}(x)  \mathbf{H}^{\frac{1}{2}} (x)\right) \overset{ \eqref{als:an:40}}{=} \frac{\rho \lambda^+_{i}\left(G^{-\frac{1}{2}} \mathbf{H}(x) G^{-\frac{1}{2}} \right)}{1+ \rho \lambda^+_{i}\left(G^{-\frac{1}{2}} \mathbf{H}(x) G^{-\frac{1}{2}} \right)} \overset{ \eqref{als:an:41}}{=} \lambda_{i}^+ \left(G^{-\frac{1}{2}} \mathbf{L}(x) G^{-\frac{1}{2}}\right).
\end{align}
These resolve from the fact that both $U$ and $V$ are orthogonal matrices. Also, both $\Sigma ( 1/\rho \ I + \Sigma^{\top}\Sigma )^{-1} \Sigma^{\top} $ and $\Sigma^{\top}  (1/\rho \ I + \Sigma\Sigma^{\top} )^{-1} \Sigma$ are diagonal matrices with $\frac{\rho \lambda^+_{i}\left(Y^{\top}Y \right)}{1+ \rho \lambda^+_{i}\left(Y^{\top}Y\right)} $ be the nonzero diagonal entries in $i^{th}$ position. This proves the Lemma.
\end{proof}

\begin{lemma}
\label{als:an:lem:11}
For any $x \in \R^n$, the following relation holds:
\begin{align*}
    \textbf{Range}(\mathbf{H}(x)) = \textbf{Range}(\mathbf{H}^{\frac{1}{2}} (x) \mathbf{K}(x)  \mathbf{H}^{\frac{1}{2}} (x)),
\end{align*}
that implies the following:
\begin{align}
 \label{als:an:43}
  \xi(x)  =  \lambda_{\min}^+ \left(\mathbf{H}^{\frac{1}{2}} (x) \mathbf{K}(x)  \mathbf{H}^{\frac{1}{2}} (x)\right) = \min_{i: \sigma_i(x) > 0}   \frac{\rho \sigma^2_i(x)}{1+ \rho \sigma^2_i(x)} =  \frac{\rho \lambda^+_{\min}\left(G^{-\frac{1}{2}} \mathbf{H}(x) G^{-\frac{1}{2}} \right)}{1+ \rho \lambda^+_{\min}\left(G^{-\frac{1}{2}} \mathbf{H}(x) G^{-\frac{1}{2}} \right)},
\end{align}
where, $\sigma_i(x) = \sqrt{\lambda_i\left(\mathbf{H}^{\frac{1}{2}}(x) G^{-1}\mathbf{H}^{\frac{1}{2}}(x) \right)} = \sqrt{\lambda_i\left(G^{-\frac{1}{2}} \mathbf{H}(x) G^{-\frac{1}{2}} \right)} $.
\end{lemma}

\begin{proof}
Since, $\mathbf{K}(x) \succ 0$, we have the following relation:
\begin{align}
 \label{als:an:44}
 \textbf{Null}(\mathbf{K}(x))   \subset \textbf{Null}(\mathbf{H}(x)) = \textbf{Null}(\mathbf{H}^{\frac{1}{2}}(x)).
\end{align}
Now, applying Lemma \ref{als:an:lem:6} with $N(x) = \mathbf{K}(x)$ and $M(x) = \mathbf{H}^{\frac{1}{2}}(x)$, we get the following:
\begin{align}
 \label{als:an:45}
   \textbf{Range}(\mathbf{H}^{\frac{1}{2}}(x)) =    \textbf{Range}(\mathbf{H}^{\frac{1}{2}}(x) \mathbf{K}(x) \mathbf{H}^{\frac{1}{2}}(x)) = \textbf{Range}(\mathbf{H}^{\frac{1}{2}} (x) \mathbf{K}(x)  \mathbf{H}^{\frac{1}{2}} (x)).
\end{align}
Then, from the definition of $\xi(x)$, we have
\begin{align}
 \label{als:an:46}
 \xi(x) = & \min_{v \in \textbf{Range}(\mathbf{H}(x))} \frac{\langle \mathbf{H}^{\frac{1}{2}} (x) \mathbf{K}(x)  \mathbf{H}^{\frac{1}{2}} (x) v, v \rangle}{\|v\|^2} \nonumber \\
 & \overset{ \eqref{als:an:45}}{=} \min_{v \in \textbf{Range}(\mathbf{H}^{\frac{1}{2}} (x) \mathbf{K}(x)  \mathbf{H}^{\frac{1}{2}} (x))} \frac{\langle \mathbf{H}^{\frac{1}{2}} (x) \mathbf{K}(x)  \mathbf{H}^{\frac{1}{2}} (x) v, v \rangle}{\|v\|^2} = \lambda_{\min}^+ \left(\mathbf{H}^{\frac{1}{2}} (x) \mathbf{K}(x)  \mathbf{H}^{\frac{1}{2}} (x)\right).
\end{align}
Considering the singular value decomposition provided in the previous Lemma, we get the complete result.
\end{proof}

\begin{lemma}
\label{als:an:lem:12}
Assume, assumption \ref{als:an:a:2} holds. Then for any $x \in \R^n$, the following relation holds:
\begin{align}
\label{als:an:47}
     \| \nabla f(x)\|^2_{\mathbf{K}(x)} \geq  \ \xi(x) \ \| \nabla f(x)\|^2_{\mathbf{H}^{\dagger}(x)} \geq \ \xi \ \| \nabla f(x)\|^2_{\mathbf{H}^{\dagger}(x)}.
\end{align}
\end{lemma}

\begin{proof}
Since, $\textbf{Range}(\mathbf{H}(x)) = \textbf{Range}(\mathbf{H}^{\frac{1}{2}}(x))$ using the relation $\nabla f(x) \in \textbf{Range}(\mathbf{H}(x))$ (assumption \ref{als:an:a:2}), we have the following:
\begin{align}
\label{als:an:48}
 \mathbf{H}^{\frac{1}{2}}(x)   \mathbf{H}^{\dagger/2}(x)   \nabla f(x) = \nabla f(x),
\end{align}
here, we used the notation $\mathbf{H}^{\dagger/2}(x)$ to denote the quantity $\left(\mathbf{H}^{\dagger}(x)\right)^{\frac{1}{2}}$. Then using the above relation, we have
\begin{align*}
 \nabla f(x)^{\top} \mathbf{K}(x) \nabla f(x) & \overset{ \eqref{als:an:48}}{=}  \nabla f(x)^{\top}  \mathbf{H}^{\dagger/2}(x)  \mathbf{H}^{\frac{1}{2}}(x)  \mathbf{K}(x)   \mathbf{H}^{\frac{1}{2}}(x)  \mathbf{H}^{\dagger/2}(x)  \nabla f(x) \\
 & = \nabla f(x)^{\top}  \mathbf{H}^{\dagger/2}(x)   \mathbf{H}^{\frac{1}{2}} (x) \mathbf{K}(x)  \mathbf{H}^{\frac{1}{2}} (x) \mathbf{H}^{\dagger/2}(x)  \nabla f(x) \\
 & \overset{ \eqref{als:an:43}}{\geq}   \xi(x) \ \| \nabla f(x)\|^2_{\mathbf{H}^{\dagger}(x)} \overset{ \eqref{als:an:39}}{\geq} \xi \ \| \nabla f(x)\|^2_{\mathbf{H}^{\dagger}(x)},
\end{align*}
here, we used the definition of $\xi(x), \ \xi$ along with the relation $\mathbf{H}^{\dagger/2}(x)   \nabla f(x) \in \textbf{Range}(\mathbf{H}(x))$. 
\end{proof}

\begin{assumption}
\label{als:an:a:3}
Let, $x_k$ be the random iterate generated by the ANM algorithm and for all $k \geq 1$, assume one of the following holds:
\begin{align}
\label{als:an:49}
    \text{Either} \quad \mathbf{H}(x_k) \succ 0 \quad \text{or} \quad G (x_k-x_{k-1}) \in \textbf{Range}(\mathbf{H}(x_k)).
\end{align}
\end{assumption}

\begin{lemma}
\label{als:an:lem:13}
Assume assumption \ref{als:an:a:3} holds. Then we have the following:
\begin{align}
\label{als:an:50}
    \|x_k-x_{k-1}\|^2_{\mathbf{L}(x_k)} \geq \lambda_{\min}^+ \left(G^{-\frac{1}{2}} \mathbf{L}(x_k) G^{-\frac{1}{2}}\right)  \|x_k-x_{k-1}\|^2_G =  \xi(x_k) \ \|x_k-x_{k-1}\|^2_G.
\end{align}
\end{lemma}

\begin{proof}
First, assume $\mathbf{H}(x_k) \succ 0$. This implies $G^{-\frac{1}{2}} \mathbf{L}(x_k) G^{-\frac{1}{2}} \succ 0$. Since, $G^{-\frac{1}{2}} \mathbf{L}(x_k) G^{-\frac{1}{2}} \succ 0$ we have $\lambda_{\min} \left(G^{-\frac{1}{2}} \mathbf{L}(x_k) G^{-\frac{1}{2}}\right) = \lambda_{\min}^+ \left(G^{-\frac{1}{2}} \mathbf{L}(x_k) G^{-\frac{1}{2}}\right) = \xi(x_k) $. Then we have 
\begin{align}
\label{als:an:51}
    \|x_k-x_{k-1}\|^2_{\mathbf{L}(x_k)} & = (x_k-x_{k-1})^{\top}   G^{\frac{1}{2}} G^{-\frac{1}{2}} \mathbf{L}(x_k) G^{-\frac{1}{2}} G^{\frac{1}{2}} (x_k-x_{k-1}) \nonumber \\
    & \geq \lambda_{\min} \left(G^{-\frac{1}{2}} \mathbf{L}(x_k) G^{-\frac{1}{2}}\right)  \|x_k-x_{k-1}\|^2_G =  \xi(x_k) \|x_k-x_{k-1}\|^2_G.
\end{align}
Considering the second condition of assumption \ref{als:an:a:3}, we have
\begin{align}
\label{als:an:52}
 G (x_k-x_{k-1}) \in \textbf{Range}(\mathbf{H}(x_k)) = \textbf{Range}(\mathbf{H}^{\frac{1}{2}}(x_k)).
\end{align}
As $\textbf{Null} \left(\left(\frac{1}{\rho}+ \mathbf{H}^{\frac{1}{2}} (x_k)  G^{-1}  \mathbf{H}^{\frac{1}{2}} (x_k) \right)^{-1} \right) \subset\textbf{Null}(G^{-\frac{1}{2}}\mathbf{H}^{\frac{1}{2}}(x_k))$, applying Lemma \ref{als:an:lem:6} with $N(x_k) = \left(\frac{1}{\rho}+ \mathbf{H}^{\frac{1}{2}} (x_k)  G^{-1}  \mathbf{H}^{\frac{1}{2}} (x_k) \right)^{-1}$ and $M(x_k) = \mathbf{H}^{\frac{1}{2}}(x_k) G^{-\frac{1}{2}}$, we get the following:
\begin{align}
\label{als:an:53}
   \textbf{Range}(G^{-\frac{1}{2}}\mathbf{H}^{\frac{1}{2}}(x_k)) =    \textbf{Range}(G^{-\frac{1}{2}}\mathbf{H}^{\frac{1}{2}}(x_k) N(x_k) \mathbf{H}^{\frac{1}{2}}(x_k)G^{-\frac{1}{2}}) = \textbf{Range}(G^{-\frac{1}{2}} \mathbf{L}(x_k) G^{-\frac{1}{2}}).
\end{align}
This implies
\begin{align}
\label{als:an:54}
    G^{\frac{1}{2}} (x_k-x_{k-1}) \overset{ \eqref{als:an:52}}{\in}    \textbf{Range}(G^{-\frac{1}{2}}\mathbf{H}^{\frac{1}{2}}(x_k)) \overset{ \eqref{als:an:53}}{=} \textbf{Range}(G^{-\frac{1}{2}} \mathbf{L}(x_k) G^{-\frac{1}{2}}).
\end{align}
Using the Courant-Fisher Theorem along with the above relation, we get the following identity:
\begin{align}
\label{als:an:55}
    \|x_k-x_{k-1}\|^2_{\mathbf{L}(x_k)} & = (x_k-x_{k-1})^{\top}   G^{\frac{1}{2}} G^{-\frac{1}{2}} \mathbf{L}(x_k) G^{-\frac{1}{2}} G^{\frac{1}{2}} (x_k-x_{k-1}) \nonumber \\
    & \overset{ \eqref{als:an:54}}{\geq} \lambda_{\min}^+ \left(G^{-\frac{1}{2}} \mathbf{L}(x_k) G^{-\frac{1}{2}}\right)  \|x_k-x_{k-1}\|^2_G \nonumber \\
    & \overset{ \text{Lemma} \ \ref{als:an:lem:10} +  \eqref{als:an:39}}{=}  \xi(x_k) \ \|x_k-x_{k-1}\|^2_G.
\end{align}
Combining \eqref{als:an:51} and \eqref{als:an:55}, we get the required result. 
\end{proof}

\begin{theorem}
\label{als:an:thm:1}
Assume, $x_k$ are random iterates of the PNM algorithm. Then, the sequences $x_k$ converges and the following identity holds: 
\begin{align}
\label{als:an:56}
   & f(x_{k+1}) - f^*   \leq (1-\eta)^{k+1} \ [f(x_{0}) - f^*],
\end{align}
here, $\beta = \min_{x \in \mathcal{Q}_{\rho}} \lambda_{\min} \left(\mathbf{K}^{\frac{1}{2}}(x) G \mathbf{K}^{\frac{1}{2}}(x)\right)$ and $\eta = \frac{\mu \xi (\beta + \rho)}{\rho L}  $.
\end{theorem}

\begin{proof}
Using the update formula of the PNM method provided in \eqref{als:an:19} along with assumption \ref{als:an:a:1} we have,
\allowdisplaybreaks{\begin{align}
\label{als:an:57}
     f(x_{k+1}) - f^*  & \overset{ \eqref{als:an:7}}{\leq} f(x_k) -f^* + \langle \nabla f(x_k), x_{k+1}-x_k\rangle   + \frac{L}{2} \|x_{k+1}-x_k\|^2_{ \mathbf{H}(x_k)} \nonumber \\
    & \overset{ \eqref{als:an:19}}{=} f(x_k) -f^* - \frac{1}{L} \| \nabla f(x_k)\|^2_{\mathbf{K}(x_k)} + \frac{L}{2} \big \| \frac{1}{L} \mathbf{K}(x_k)  \nabla f(x_k)  \big \|^2_{ \mathbf{H}(x_k)} \nonumber \\
    & = f(x_k) -f^* - \frac{1}{L} \| \nabla f(x_k)\|^2_{\mathbf{K}(x_k)} + \frac{1}{2L} \nabla f(x_k)^{\top} \mathbf{K}(x_k)  \mathbf{H}(x_k) \mathbf{K}(x_k)  \nabla f(x_k) \nonumber \\
    & \overset{ \eqref{als:an:15}}{=} f(x_k) -f^* - \frac{1}{L} \| \nabla f(x_k)\|^2_{\mathbf{K}(x_k)} +  \frac{1}{2L} \|  \nabla f(x_k)\|^2_{ \mathbf{K}(x_k) } \nonumber \\
    & \qquad \qquad \qquad - \frac{1}{2 \rho L} \| \nabla f(x_k)\|^2_{ \mathbf{K}(x_k) G \mathbf{K}(x_k)  } \nonumber \\
    & = f(x_k) -f^* - \frac{1}{2L} \| \nabla f(x_k)\|^2_{\mathbf{K}(x_k)}  - \frac{1}{2 \rho L} \| \nabla f(x_k)\|^2_{ \mathbf{K}(x_k) G \mathbf{K}(x_k)} \nonumber \\
     & \leq f(x_k) -f^* - \frac{1}{2L} \| \nabla f(x_k)\|^2_{\mathbf{K}(x_k)} - \frac{\beta}{2 \rho L} \| \nabla f(x_k)\|^2_{ \mathbf{K}(x_k) } \nonumber \\
     & = f(x_k) -f^* - \frac{\beta + \rho}{2\rho L}  \| \nabla f(x_k)\|^2_{\mathbf{K}(x_k)} \nonumber \\
    & \overset{ \eqref{als:an:47}}{\leq} f(x_k) -f^* - \frac{\xi (\beta + \rho)}{2\rho L}    \| \nabla f(x_k)\|^2_{\mathbf{H}^{\dagger}(x_k)} \nonumber \\
     & \overset{ \eqref{als:an:10}}{\leq}  f(x_k) -f^* -   \frac{\mu \xi (\beta + \rho)}{\rho L}    [f(x_k) -f^* ] \nonumber \\
     & = (1-\eta)  [f(x_k) -f^* ] \nonumber \\
     & \leq (1-\eta)^{k+1} \ [f(x_{0}) - f^*],
\end{align}}
here, $\beta = \min_{x \in \mathcal{Q}_{\rho}} \lambda_{\min} \left(\mathbf{K}^{\frac{1}{2}}(x) G \mathbf{K}^{\frac{1}{2}}(x)\right)$ and $\eta = \frac{\mu \xi (\beta + \rho)}{\rho L}  $. Unrolling the recurrence, we get the required result.
\end{proof}

\begin{theorem}
\label{als:an:thm:2}
Assume, $x_k$ is the random iterates of the ANM algorithm. Then, the sequence $x_k$ converges and the following identity holds: 
\begin{align}
\label{als:an:58}
   & \mathcal{V}_{k+1}(\rho)  \leq \left(1- \xi \frac{ \mu }{L}\right)^{k} \ \mathcal{V}_1(\rho).
\end{align}
\end{theorem} 

\begin{proof}
Using the update formula of the ANM method, we have 

From the first identity of assumption \ref{als:an:a:1} we have,
\begin{align}
\label{als:an:59}
    f(x_{k+1}) & \overset{ \eqref{als:an:7}}{\leq} f(x_k) + \langle \nabla f(x_k), x_{k+1}-x_k\rangle   + \frac{L}{2} \|x_{k+1}-x_k\|^2_{ \mathbf{H}(x_k)}.
\end{align}
Now, we use the update formula of the ANM algorithm provided in \eqref{als:an:26} to simplify the right hand side of \eqref{als:an:59}.Using \eqref{als:an:26}, the second term of \eqref{als:an:59} can be simplified as follows:
\begin{align}
    \label{als:an:60}
  \langle \nabla f(x_k),  x_{k+1}-x_k\rangle  & \overset{ \eqref{als:an:26}}{=} - \frac{1}{\rho} \nabla f(x_k)^{\top}  \mathbf{K}(x_k) \left[\frac{\rho}{L} \nabla f(x_k) - G (x_k-x_{k-1})\right] \nonumber \\ 
  & = - \frac{1}{L} \| \nabla f(x_k)\|^2_{\mathbf{K}(x_k)} + \frac{1}{\rho} (x_k-x_{k-1})^{\top} G \mathbf{K}(x_k)  \nabla f(x_k).
\end{align}
The third term of \eqref{als:an:59} can be simplified as follows:
\allowdisplaybreaks{\begin{align}
\label{als:an:61}
    \frac{L}{2} \|x_{k+1}-x_k\|^2_{ \mathbf{H}(x_k)} & \overset{ \eqref{als:an:26}}{=}  \frac{L}{2} \big \| \frac{1}{L} \mathbf{K}(x_k)  \nabla f(x_k) - \frac{1}{\rho}  \mathbf{K}(x_k)  G (x_k-x_{k-1}) \big \|^2_{ \mathbf{H}(x_k)} \nonumber \\
    & = \frac{1}{2L} \|\mathbf{K}(x_k)  \nabla f(x_k)\|^2_{   \mathbf{H}(x_k)  }  + \frac{L}{2\rho^2} \|\mathbf{K}(x_k) G (x_k-x_{k-1})\|^2_{  \mathbf{H}(x_k)   } \nonumber \\
    & \qquad \qquad - \frac{1}{\rho} \langle G (x_k-x_{k-1}), \nabla f(x_k) \rangle_{\mathbf{K}(x_k)   \mathbf{H}(x_k)  \mathbf{K}(x_k)}   \nonumber \\
    &  \overset{ \eqref{als:an:15}}{=}  \frac{1}{2L} \|  \nabla f(x_k)\|^2_{ \mathbf{K}(x_k) } - \frac{1}{2 \rho L} \| \nabla f(x_k)\|^2_{ \mathbf{K}(x_k) G \mathbf{K}(x_k)} \nonumber \\
    & + \frac{L}{2 \rho^2} \|G(x_k-x_{k-1})\|^2_{\mathbf{K}(x_k)}     - \frac{L}{2 \rho^3} \|G(x_k-x_{k-1})\|^2_{\mathbf{K}(x_k) G \mathbf{K}(x_k)} \nonumber \\
    & - \frac{1}{\rho} (x_k-x_{k-1})^{\top} G \mathbf{K}(x_k)  \nabla f(x_k) \nonumber \\
    & + \frac{1}{\rho^2} (x_k-x_{k-1})^{\top} G  \mathbf{K}(x_k) G \mathbf{K}(x_k) \nabla f(x_k).
\end{align}}
Similarly, we have
\begin{align}
\label{als:an:62}
    \frac{L}{2 \rho}  \|x_{k+1}-x_k\|^2_G  & \overset{ \eqref{als:an:26}}{=} \frac{L}{2 \rho} \big \| \frac{1}{L} \mathbf{K}(x_k)   \nabla f(x_k) - \frac{1}{\rho} \mathbf{K}(x_k) G (x_k-x_{k-1})\big \|^2_G \nonumber \\
    & = \frac{1}{2 \rho L} \|  \nabla f(x_k)\|^2_{ \mathbf{K}(x_k) G \mathbf{K}(x_k)}  + \frac{L}{2 \rho^3}  \|G(x_k-x_{k-1})\|^2_{\mathbf{K}(x_k) G \mathbf{K}(x_k)   } \nonumber \\
    & \qquad \qquad - \frac{1}{\rho^2} (x_k-x_{k-1})^{\top} G   \mathbf{K}(x_k) G \mathbf{K}(x_k)  \nabla f(x_k).
\end{align}
Substituting \eqref{als:an:60} and \eqref{als:an:61} back in \eqref{als:an:59} and using \eqref{als:an:62} we have
\allowdisplaybreaks{\begin{align}
\label{als:an:63}
  \mathcal{V}_{k+1}(\rho) &   = f(x_{k+1}) -f^* + \frac{L}{2 \rho}  \|x_{k+1}-x_k\|^2_G \nonumber \\
   & \leq f(x_k) -f^*  + \langle \nabla f(x_k), x_{k+1}-x_k\rangle   + \frac{L}{2} \|x_{k+1}-x_k\|^2_{ \mathbf{H}(x_k)} + \frac{L}{2 \rho} \|x_{k+1}-x_k\|^2_G \nonumber \\
   & \overset{ \eqref{als:an:60} + \eqref{als:an:61} + \eqref{als:an:62}}{=} f(x_k) -f^*  - \frac{1}{L} \| \nabla f(x_k)\|^2_{\mathbf{K}(x_k)} + \cancel{ \frac{1}{\rho} (x_k-x_{k-1})^{\top} G \mathbf{K}(x_k)  \nabla f(x_k)} \nonumber \\
  &  +  \frac{1}{2L} \|  \nabla f(x_k)\|^2_{ \mathbf{K}(x_k) } - \cancel{\frac{1}{2 \rho L} \| \nabla f(x_k)\|^2_{ \mathbf{K}(x_k) G \mathbf{K}(x_k)}} \nonumber \\
  & - \cancel{\frac{L}{2 \rho^3} \|G(x_k-x_{k-1})\|^2_{\mathbf{K}(x_k) G \mathbf{K}(x_k)}}  + \frac{L}{2 \rho^2} \|G(x_k-x_{k-1})\|^2_{\mathbf{K}(x_k)} \nonumber \\
    &    -  \cancel{\frac{1}{\rho} (x_k-x_{k-1})^{\top} G \mathbf{K}(x_k)  \nabla f(x_k)} + \cancel{\frac{1}{\rho^2} (x_k-x_{k-1})^{\top} G  \mathbf{K}(x_k) G \mathbf{K}(x_k) \nabla f(x_k)} \nonumber \\ 
  & + \cancel{\frac{1}{2 \rho L} \|  \nabla f(x_k)\|^2_{ \mathbf{K}(x_k) G \mathbf{K}(x_k)}}  + \cancel{\frac{L}{2 \rho^3}  \|G(x_k-x_{k-1})\|^2_{\mathbf{K}(x_k) G \mathbf{K}(x_k)   } } \nonumber \\
  & - \cancel{\frac{1}{\rho^2} (x_k-x_{k-1})^{\top} G   \mathbf{K}(x_k) G \mathbf{K}(x_k)  \nabla f(x_k)}  \nonumber \\
   & \leq f(x_k) -f^* - \frac{1}{2L} \|\nabla f(x_k)\|^2_{ \mathbf{K}(x_k) } + \frac{L}{2 \rho^2} \|G (x_k-x_{k-1})\|^2_{ \mathbf{K}(x_k)} \nonumber \\
   & \overset{ \eqref{als:an:16}}{\leq} f(x_k) -f^* - \frac{1}{2L} \|\nabla f(x_k)\|^2_{ \mathbf{K}(x_k) } +  \frac{L}{2 \rho} \|x_k-x_{k-1}\|^2_G - \frac{L}{2 \rho} \|x_k-x_{k-1}\|^2_{\mathbf{L}(x_k)} \nonumber \\
   & \overset{ \eqref{als:an:47} + \eqref{als:an:50}}{\leq} \mathcal{V}_k(\rho)  - \frac{ \xi(x_k)}{2L} \|\nabla f(x_k)\|^2_{\mathbf{H}^{\dagger}(x_k)}  - \frac{\xi(x_k) L}{2 \rho } \|x_k-x_{k-1}\|^2_G \nonumber \\
   &  \overset{ \eqref{als:an:10}}{\leq} \mathcal{V}_k(\rho)   - \frac{ \mu \xi(x_k)}{L} \left[f(x_k)-f^*\right]  - \frac{\xi(x_k) L}{2 \rho } \|x_k-x_{k-1}\|^2_G \nonumber \\
    &  \overset{ \eqref{als:an:39}}{\leq} \mathcal{V}_k(\rho)   - \frac{ \mu \xi }{L} \left[f(x_k)-f^*\right]  - \frac{\xi L}{2 \rho } \|x_k-x_{k-1}\|^2_G \nonumber \\
   & \leq (1-\theta) \ \mathcal{V}_k(\rho)   \nonumber \\
   & = \left(1- \xi \frac{ \mu }{L}\right) \ \mathcal{V}_k(\rho)   \nonumber \\
   & \leq \left(1- \xi \frac{ \mu }{L}\right)^{k} \ \mathcal{V}_1(\rho),
\end{align}}
where, $\theta = \xi \min \{1, \frac{ \mu }{L}\} = \xi \frac{ \mu }{L}$. This proves the Theorem.
\end{proof}


\section{Conclusion}
\label{als:an:sec:concl}

In this work, we proposed two variants of the Newton method for solving the unconstrained minimization problem. Our proposed approach incorporates the penalty method and augmented Lagrangian techniques to the Newton system. We provided convergence results of the proposed method under mild conditions imposed on the function $f(x)$. The proposed penalty Newton method generalizes the so-called \textit{Levenberg} and \textit{Levenberg-Marquardt} algorithm into one framework. The proposed algorithms outperform the existing Newton method on various test instances. Building on our work one can design penalty and augmented variants of the so-called \textit{Randomized Subspace Newton} (RSN) method proposed in \cite{gower2019rsn}.

\section*{Acknowledgements}
The Author would like thank Rober M. Gower for his thoughtful comments on an earlier version of the manuscript.

\bibliographystyle{siamplain}
\bibliography{references}

\end{document}